\documentclass[10pt,twoside]{amsart}
\usepackage{amssymb}
\usepackage{latexsym}

\usepackage{version}
\usepackage{amsmath}
\usepackage{color}
\usepackage{amsthm}
\usepackage{amscd}
\usepackage{amsfonts}
\usepackage{graphics}

\theoremstyle{definition}
\newtheorem*{theorem*}{Theorem}
\newtheorem{theorem}{Theorem}[section]
\newtheorem{definition}{Definition}[section]
\newtheorem{lemma}{Lemma}[section]

\newtheorem{remark}{Remark}[section]

\newtheorem{corollary}{Corollary}[section]
\newtheorem*{ack}{Acknowledgments}

\renewcommand{\theequation}{I.\arabic{equation}}

\catcode`\@=12

\long\def\salta#1{\relax}

\def\rife#1{(\ref{#1})} 
 
\def\eps{\varepsilon}

\def\dys{\displaystyle}

\def\l2h10{L^2 (0,T ; H^1_0 ( \Omega ))}

\def\bc{\begin{cases}} 
\def\ec{\end{cases}} 
\def\be{\begin{equation}} 
\def\ee{\end{equation}}

\def\t12{\mathcal T^{1,2} (\Omega)}
\def\farc{\frac}

\author[L.Gonella]{  Laura Gonella }
  \address{  Departamento de An\'alisis Matem\'atico - Universidad de Granada, Campus Fuentenueva S/N, 18071,
  Granada, Spain} 
\email{lauragonella@ugr.es}
\thanks{Research supported by MICINN Ministerio de Ciencia e Innovaci\'on (Spain) MTM2009-10878 and Junta de Andaluc\'{\i}a FQM-116.}

\date{}

\title[ Periodic Solutions for  
Singular  equations with strong resonant data]{   Existence of Periodic Solutions for some \\
Singular Elliptic equations with\\ strong resonant data}



\begin{document}
\maketitle

\renewcommand{\theequation}{\arabic{section}.\arabic{equation}}

\begin{center}
{\bf\small Abstract}
\end{center}
We prove the existence of at least one $T$-periodic solution $(T>0)$ for differential equations of the form
$$\Bigg(\frac{u'(t)}{\sqrt{1-u'^2(t)}}\Bigg)' =f(u(t))+h(t),\qquad \text{ in } (0,T),$$
where 
$f$ is a continuous function defined on $\mathbb{R}$ that satisfies a {\it strong resonance condition}, $h$  is  continuous and with zero mean value. Our method uses variational techniques for nonsmooth functionals.

\section{Introduction}
In this paper we deal with existence of periodic solutions for a class of equations whose model is 
\begin{eqnarray}\label{primoo}
\Bigg(\frac{u'(t)}{\sqrt{1-u'^2(t)}}\Bigg)'=f(u(t))+h(t), \,\quad  t\in (0,T),%\\
\end{eqnarray}
where $T>0$, $f$ and $h$ are continuous and $\int_0^T h(t)dt=0$ (i.e. $h$ has zero mean value). This equation is known in literature as \emph{relativistic} forced pendulum.\smallskip

We recall that the equation of a   
 {\it classical} forced pendulum provided with periodic boundary conditions can be formulated as 
\begin{eqnarray}\label{classicBC}
\begin{cases}
u''(t)=A\sin\big(u(t)\big)+h(t), \,\quad  t\in (0,T),\\%\\
u(0)-u(T)=0=u'(0)-u'(T),
\end{cases}
\end{eqnarray}
where $A\in\mathbb{R}$ and $h$ is a  continuous function with zero mean value.
{%\color{blue}
From the physical point of view, $u(t)$ represents the position of the pendulum and $u''(t)$ its acceleration. The equation in \rife{classicBC} is indeed the expression of the classical Newton's law,
where the external forces are represented by $h(t)$, the temporal-dependent part, and by $A\sin\big(u(t)\big)$, which instead depends on the position. 
}

Problem \rife{classicBC} has been studied, among others, by Hamel (see the pioneering paper \cite{Hamel}), and later 
by Willem (see \cite{will}) and Dancer in \cite{dancer}. These authors proved the existence of at least a solution by minimizing  the energy functional associated to the equation in suitable spaces. Due to the periodicity of this functional, 
once found the minimum $u$, we know infinitely many other minima to exist, which are of the form $u+j\omega$, $\forall j \in \mathbb{Z}$. Later on, 
Mawhin and Willem (see \cite{mawwil}) %proved 
 exploted this property in order to prove the existence of a second geometrically distinct solution (i.e. a solution $v$ of the problem, such that  $v\not\equiv u+j\omega,\,\, j\in\mathbb{Z}$) which is not a minimum anymore, but has a different nature.

Another contributions were given by Thews (\cite{teps}), by Ambrosetti and Coti Zelati in \cite{ambcz}, by Coti Zelati in \cite{cotizel} and by Arcoya  (\cite{arcca}), who studied problem \rife{classicBC} replacing the external force $\sin(u)$ with a continuous function $f(u)$, under an assumption on its behavior at infinity, namely: 
\begin{eqnarray*}
\lim_{\vert s\vert\to+\infty}f(s)=\lim_{\vert s\vert\to+\infty}F(s)=0,
\end{eqnarray*}
where by $F$ we mean the primitive of $f$ defined by $ F(s)=\int_0^s f(\sigma)d\sigma$.
%He found some conditions under which the problem has at least a solution, using again a variational approach.
%
%Results concerning the existence of solutions in this last case, can be found in .\smallskip
In these papers one can found conditions under which the problem is solvable and informations about multiplicity of solutions.\smallskip

The motivation to generalize the classical pendulum equation 
comes from the study of a pendulum moved by the  laws of relativity. It is well known (see for instance \cite{fis}) that the relativistic Newton's  law, obtained by use of Lorentz transfomation $\gamma$, is  
$$
F = m \farc{d }{d t }  \gamma (u' (t)) \,,
\quad \mbox{ 
where }
\quad \gamma (s)  = \frac{s}{\sqrt{1- \frac{s^2}{c^2}} },
$$
$u'(t)$ is the {velocity} and $c$ is the speed of light in vacuum. Thus \rife{primoo} represents the equation of a forced pendulum in the relativistic framework (by considering $m=c=1$). Moreover, due to the nature of the problem, it seems natural to look for periodic solutions, which leads us to the study of
\begin{eqnarray}\label{genI}
\begin{cases}
\Big(\frac{u'(t)}{\sqrt{1-u'^2(t)}}\Big)'=f(u(t))+h(t), \,\quad  t\in (0,T),\\[1.5 ex]%\\
u(0)-u(T)=0=u'(0)-u'(T),\smallskip
\end{cases}
\end{eqnarray}
where $f$ is a  
continuous function with primitive $ F(s)=\int_0^s f(\sigma)d\sigma$.  

The main difficulties in dealing with this problem are due both on the presence of a nonlinearity in the right hand side of the equation and on the singularity of the principal part of the operator, which makes the energy functional nonsmooth. Let us note that, due to the fact that we handle functions with bounded derivative, the natural framework in which to work is the space of Lipschitz functions, with Lipschitz constant less than $1$.
The functional is defined as
$$\dys I(u)=\int_0^T \Big[1-\sqrt{1-u'^2}+F(u)+h(t)u\Big]dt, \quad\text{ if } 
\Vert u'\Vert_\infty\leq1,$$
and as $+\infty$ otherwise. 

The lack of regularity of $I$ makes it necessary to use Szulkin's theory (see \cite{szul}) which allows us to define critical points for such a nonsmooth functional (see Definition \ref{defptocrit}) and which makes the study more complicated (notice that, in contrast with the classical case, a critical point does not satisfy a family of identities, but a family of inequalities). Szulkin's theory provides as well an adapted version of the Palais-Smale condition (see Definition \ref{psseq}) and of the {Mountain-Pass theorem (see Theorem \ref{saddle}).} \medskip

Recently Brezis and Mawhin \cite{due} proved the existence of a solution of \rife{genI}, in the case of $2\pi$-periodic $f$ with zero mean value (or equivalently with a $2\pi$-periodic primitive $F$) by minimizing the energy functional associated to the equation.  
 They used the direct method of calculus of variations to guarantee that a minimum exists, by proving the boundedness from below and the lower semicontinuity of the functional.  They also proved that such a minimum (actually their argument can be extended to any critical point, see Theorem \ref{strongexist}) is a solution of the problem (in sense of Definition \ref{defsol}), so that they obtained the existence of infinitely many solutions (since, as in the classical framework, the energy functional is $2\pi$-periodic).

%  Bereanu and Torres in \cite{secsol} extend the existence result of Mawhin and Willem \cite{mawwil} for problem \rife{classicBC} to problem \rife{genI}. They apply Szulkin's  version of the Mountain Pass theorem  
%  to a truncation of the functional (which satisfies the Palais-Smale condition) to get the existence of either a third minimum or a different nature critical point lying between two minima  found by Brezis and Mawhin. 
Bereanu and Torres in \cite{secsol} extend the existence result of Mawhin and Willem \cite{mawwil} for problem \rife{classicBC} to problem \rife{genI}: indeed they prove  the existence of a critical point that is not a translation of the  minimum  found by Brezis and Mawhin. This imply the existence of a second family of critical points that can be either minima or Mountain Pass (in this latter case they use Szulkin's version of Mountain Pass theorem).\smallskip

The aim of this paper is to extend problem \rife{genI} replacing the periodicity of $f$ with a different nonlinearity. 
We consider an hypothesis that is called in literature as {\it strong resonance condition}, i.e. $f(s)$ is a real continuous function with primitive $ F(s)=\int_0^s f(\sigma)d\sigma$ satisfying 
\begin{equation}
\begin{array}{c}\label{newhp}
\displaystyle\lim_{\vert s\vert\to+\infty}f(s)=0,\quad 
\displaystyle\lim_{\vert s\vert\to+\infty}F(s)=\alpha,
\end{array}
\end{equation}
for some $\alpha\in\mathbb{R}$.
{
This assumption allows us to recover some compactness condition for the functional which was guaranteed in case of a periodic $f$,  and which is necessary to prove the existence of a critical point.\medskip

Therefore, 
our main result is the following. \smallskip

{\it
Under hypothesis \rife{newhp}, problem \rife{genI} has at least one solution.}\smallskip

%\end{theorem*} 
Actually, we prove this existence result for a more general version of problem \rife{genI}, which can be found in Theorem \ref{new}.
In order to prove it, we consider the energy functional associated to the equation of \rife{genI}.{
We first prove that it satisfies the Palais-Smale condition at every level but one %(we shall call it $c$)
 and that it achieves its infimum value on a suitable subspace $W$ of its domain. Two cases may occur: either there exists a function in $W$ in which the functional has a value less or equal than the critical level, or the functional, evaluated on $W$, always lies above of it. We are able to prove that in both cases a critical point exists, which can be a minimum (if the first case occurs) or a mountain-pass nature critical point (if the second one stands).  Applying the regularity result in \cite{due}, we prove that such a critical point is a solution of \rife{genI}.  }

Observe that the above proof provides an alternative: the solution may be a minimizer of the functional, or a mountain pass critical point. In order to give additional informations about this uncertainty, we provide sufficient conditions under which it is ensured  the existence of either a minimum  (Corollary \ref{cor1}) or a mountain pass (Corollary \ref{cor2}).
\medskip

%\section{dfgdfg}
\setcounter{equation}{0}

\section{Hypotheses and statement of the results}
%The aim of this chapter is to prove an existence result for solutions of our problem  in the case of non-periodic data satisfying suitable conditions. 
%As before, the approach will be variational.
Let us consider the following problem:
%In this case we prove that the functional associated to the problem possesses either a minimum, or it has a critical point of  mountain  pass type.
\begin{equation}\label{simplenew}
\begin{cases}
\displaystyle (\phi(u'(t)))' = f(t,u(t))+h(t),\qquad\qquad\qquad \text{ in } (0,T);\\%[1.5 ex]
u(0)-u(T)=0=u'(0)-u'(T).
\end{cases}
\end{equation}
%where suitable hypotheses on the functions involved are in order to be given.
Let $a>0$ and let us assume the following hypotheses on the functions involved.\smallskip

\begin{itemize}
\item[$(H_\Phi)$] There exists $\Phi:[-a,a] \rightarrow \mathbb{R}$ such that $\Phi\in C[-a,a]\cap C^1(-a,a)$ and $\phi := \Phi':(-a,a) \rightarrow \mathbb{R}$ is an increasing homeomorphism such that $\phi(0)=0$.\\
%Let us note that under these assumptions, it stands:
%\begin{equation}\label{monot}
%\phi(s)s>0,\,\,\forall s\in(-a,a)\setminus\{0\};
%\end{equation}

\item[$(H_f)$]   $f: [0,T] \times \mathbb{R} \to \mathbb{R}$ is a continuous function such that  
$$
\lim_{\vert s\vert\to+\infty} f(t,s)=0 \quad \mbox{uniformly with respect to t}.
$$
$F$ is the primitive of $f$ defined as
\[
F(t,\tau)=\int_0^\tau f(t,s)ds\quad\quad\forall (t,\tau)\in[0,T]\times\mathbb{R}\,,
\]
and there exists a constant $\alpha \in\mathbb{R}$ such that: 
$$
\lim_{\vert s\vert\to+\infty } F(t,s)=\alpha \quad \mbox{uniformly with respect to t}.
$$
\item[$(H_h)$] $h:[0,T]\to\mathbb{R}$ is a continuous function such that
$$
\int_0^T h(t)dt=0.
$$
\end{itemize}

Observe that, by $(H_\Phi)$, without loss of generality we can suppose $\Phi(0)=0$.
Let us note that $(H_f)$ implies that $F$ is derivable with respect to the second variable 
and that both $f$ and $F$ are bounded, that is, there exists $C \in \mathbb{R}$ such that:
\begin{eqnarray}
%\vert f(u)\vert \leq C_1\qquad\forall u\in K\label{flim}\\
\vert f(t,\tau)\vert+\vert F(t,\tau)\vert \leq C,\qquad\forall (t,\tau)\in[0,T]\times\mathbb{R}\label{Flim}.
\end{eqnarray}

%Let us remark that this hypothesis on $f$ and its primitive, is known in literature as \emph{strong resonance} hypothesis%(see \cite)
%.\\
We need now to make it clear the sense we give to a solution of problem \rife{simplenew}.
\begin{definition}\label{defsol}
A solution of \rife{simplenew} is a function $u\in C^1[0,T]$ with $u(0)=u(T)$, $u'(0)=u'(T)$, such that { $\Vert u' \Vert _\infty < a$} and $\phi\circ u'$ is absolutely continuous, whose derivative satisfies %\rife{simplenew} as an identity among $L^1$-functions, which implies that $(\phi(u'))' = f(t,u)+h(t)$ stands almost everywhere.
\begin{equation*}%\label{eqIgen}
(\phi(u'))' = f(t,u)+h(t) \ \ \text{ a.e. } t\in(0,T).
\end{equation*}
%and
%\[
%u(0)-u(T)=0=u'(0)-u'(T).
%\]
\end{definition}
%\begin{remark}
%It is possible to obtain an \emph{a posteriori} information about the regularity of the solution of \rife{simplenew}. \\Since $\phi\circ u'$ is absolutely continuous,  in particular
%$\phi (u'(t)) \in C[0,T]$, which, using that $\phi$ is an homomorphism, clearly implies $ u'(t) \in C[0,T]$.\\
% In conclusion, if $u$ is a solution of \rife{simplenew}, then
%$$ u\in C^1[0,T].$$
%\end{remark}
%Notice that, as far as this definition is concerned, \rife{eqIgen} holds  almost everywhere.
Here we state our main results.
\begin{theorem}\label{new}
\emph{Under hypotheses $(H_\Phi)$, $(H_f)$ and $(H_h)$, there exists at least a solution of Problem \rife{simplenew}. }
\end{theorem}
Let us stress that we are looking for a function whose derivative is bounded. Therefore, the natural variational framework to work in is the space %of $T$-periodic Lipschitz functions 
$W_{\#}^{1,\infty} (0,T)$. 
This is the Sobolev space of $T$-periodic functions $u$ such that $u$ and its distributional derivative $u'$ are bounded, equipped with the norm
\[
\Vert u \Vert _{W_{\#}^{1,\infty}}=\Vert u \Vert _\infty +\Vert u' \Vert _\infty.
\]
%where $\Vert u \Vert _\infty =\underset{t\in[0,T]}{\mbox{ ess sup }}\vert u(t)\vert$\smallskip.

Actually we will look for the solution in a convex subset of $W_{\#}^{1,\infty} (0,T)$, namely
\[
K=\big{\{} v\in W_{\#}^{1,\infty} (0,T) \text{ such that } \Vert v'\Vert_\infty \leq a\big\}.
\]
%that solves problem \rife{simplenew}. 
\begin{remark}
If we find a solution $u$ of \rife{simplenew} in $K$, we can obtain an \emph{a posteriori} information about its regularity. In fact, since $\phi\circ u'$ is absolutely continuous,  in particular
$\phi (u'(t)) \in C[0,T]$, which, using that $\phi$ is an homomorphism, clearly implies $ u'(t) \in C[0,T]$.
 In conclusion, $ u\in C^1[0,T].$
\end{remark}

Following a standard variational procedure, we will associate a suitable functional to problem \rife{simplenew} and prove that it has at least one critical point. Next, we will prove that such a critical point is actually a solution.\smallskip
%A standard variational procedure to prove the existence of a solution of a problem, consists in introducing a suitable functional, proving that it has a critical point and then show that it is a solutions of the problem.

Let us introduce the energy functional associated to \rife{simplenew}:
\begin{equation*}\displaystyle
I(v):=
\begin{cases}
\dys \int_0^T \Big[\Phi(v')+h(t)v+F(t,v)\Big]dt , \qquad &\mbox{ if } v\in K,\\[1.5 ex]
+\infty, &\mbox{ if } v\in  W_{\#}^{1,\infty} (0,T)\setminus K.
%+ \int_0^T h(t) u(t) dt
\end{cases}
\end{equation*}

As we already pointed out, $I$ is a nonsmooth functional and it has the structure required by Szulkin's theory (see \cite{szul}), that we briefly recall here.\smallskip\\
The functional $I:W_{\#}^{1,\infty}(0,T)\rightarrow\mathbb{R}\cup\{+\infty\}$ is decomposable as
\[
I=J+\mathcal{F}, 
\]
where 
\begin{equation*}
%\begin{array}{l}
J(v)=
\begin{cases}
\dys \int_0^T \big[\Phi(v')+h(t)v\big]dt ,  &v\in K,\\%[1.5 ex]
+\infty, & v\in  W_{\#}^{1,\infty}(0,T)\setminus K.
%+ \int_0^T h(t) u(t) dt
\end{cases}
\displaystyle\text{ and }\ 
\displaystyle\mathcal{F}(v)=\int_0^T F(t,v)dt,
%\end{array}
\end{equation*}
for any $v\in W_{\#}^{1,\infty}(0,T)$. Observe that $J$ is convex, proper and lower semicontinuous with respect to the topology of $C[0,T]$ (as it can be seen by using the same argument  as in the proof of Lemma 1 of \cite{due}), and $\mathcal{F}$ is $C^1$ (it is standard to see that it has this required regularity). 

According to Szulkin's theory, we have the following definition of critical point of $I$.
\begin{definition}\label{defptocrit}
A function $u\in W_{\#}^{1,\infty}(0,T)$ is a \emph{critical point} of the functional $I$ if $u\in K$ and it satisfies the inequality
\begin{equation*}%\label{disvarcrit}
{J}(v)-{J}(u)+\langle \mathcal{F}'(u),v-u\rangle\geq0 \quad\text{ for all } v\in W_{\#}^{1,\infty}(0,T).
\end{equation*}
We say that $c\in \mathbb{R}$ is a \emph{critical value} of $I$ if there exists a critical point $v\in W_{\#}^{1,\infty}(0,T)$ such that $I(v)=c$.
\end{definition}
As we already noticed, dealing with this family of inequalities provides additional defficulties.\smallskip

The main step in proving Theorem \ref{new} is the following result, which ensures the existence of a critical point of the functional $I$.
\begin{theorem}\label{newcp}
\emph{
If assumptions $(H_\Phi)$, $(H_f)$ and $(H_h)$ hold, then there exists at least a critical point for $I$. }
\end{theorem}
%We say that $c\in \mathbb{R}$ is a \emph{critical value} if there exists a $u\in D(\mathcal{P})$ such that $I(u)=c$.\\
%Note that $u\in W_{\#}^{1,\infty}(0,T)$ can be here replaced by $u\in D(\mathcal{P})$.
%\smallskip
In order to prove it, we need to introduce the notion of Palais-Smale condition at the level $c\in \mathbb{R}$ is in this framework.% which we will consider later.
\begin{definition}\label{psseq}
For every $c\in\mathbb{R}$, a sequence $\{u_n\}\subset W_{\#}^{1,\infty}(0,T)$ is a \emph{Palais-Smale sequence} at level $c$  (in brief $(PS)_c$-sequence) if
\begin{itemize}
\item$I(u_n)= c+\eps_n$
\item${J}(v)-{J}(u_n)+\langle \mathcal{F}'(u_n),v-u_n\rangle\geq-\varepsilon_n\Vert v-u_n\Vert_{W_{\#}^{1,\infty}} \,,\quad \forall v\in W_{\#}^{1,\infty}(0,T), $
\end{itemize}
where $\varepsilon_n$ tends to $0$ as $n$ diverges.

%\end{definition}
%Let us moreover state a compactness condition that is useful to ensure existence of critical points.
%Let us moreover state a compactness condition that is useful to ensure existence of critical points.
%\begin{definition}\label{funzPS}
We say that a functional $I$ satisfies the \emph{ Palais-Smale condition at the level $c$} (in brief $(PS)_c$-condition) if any $(PS)_c$-sequence has a uniformly  convergent subsequence in $[0,T]$.
\end{definition}\smallskip

{
As we already mentioned in the Introduction, we will give further informations about the nature of the solution we prove to exist. %In order to state this result, we need to anticipate some notions.\medskip
It will be useful to decompose any  $u\in W_{\#}^{1,\infty}(0,T) $
as
\begin{equation}\label{decomp}
u(t)=\bar{u}+\tilde{u}(t),\, \text{ where }\, \bar{u}=\frac{1}{T}\int_0^T u(t)dt\, \text{ and } \, \int_0^T \tilde{u}(t)dt=0.
\end{equation}
In this way, the entire space can be decomposed as
$$W_{\#}^{1,\infty}(0,T)=V\oplus W,$$
where
$V=\{v\in W_{\#}^{1,\infty}(0,T): v \text{ is constant}\}$ and $W$ is its topological and algebraic complement.\smallskip

%We will consider the following expression of the functional $I$ over $K$:
%\[
%\forall v\in K: I(v)=J(\tilde{v})+\int_0^T F(t,v)dt,
%\]
%where $J:W\to\mathbb{R}$ and 
%\begin{equation}\label{jei}
%\dys J(\tilde{v})=\int_0^T
%\big[\Phi(\tilde{v}')+h(t)\tilde{v}\big]dt. 
%\end{equation}
%Both $I$ and $J$ are bounded and lower semicontinuous in $W\cap K$. 
We
prove (see Remark \ref{Jbound}), that there exists $\tilde{w}\in W\cap
K$ such that
\[
%\beta=\inf_{\tilde{v}\in W}I(\tilde{v})=I(\tilde{z}) \text{ and }
m=\min_{\tilde{v}\in W}J(\tilde{v})=J(\tilde{w}). %\medskip
\]
Notice that, due to the definition of the functional, a minimizer will always belong to $K$.\smallskip

%The proof of the existence of a critical point for $I$ follows several steps. We
%first prove that $(PS)_c$ holds for every constant $c\neq m+\alpha T$. Then we
%have an alternative: for any level $c\leq m+\alpha T$ we prove the existence of
%a minimum, while for any  $c> m+\alpha T$, we prove the existence of a mountain
%pass. Actually , w
We give here some sufficient conditions to have a solution by minimization or of mountain pass nature. 
\begin{corollary}\label{cor1}
If, in addition to $(H_\Phi)$, $(H_f)$ and $(H_h)$, there exists $ \bar{v}_0\in\mathbb{R} \text{ such that }$
\begin{equation}\label{condmin2}
 \int_0^T F(t,\tilde{w}+\bar{v}_0)dt\leq \alpha T,
\end{equation}
then $I$ has a minimum.
\end{corollary}
\begin{remark} If $sf(t,s)\geq0$ for all $(t,s)\in[0,T]\times \mathbb{R}$, then $F(t,s)\leq \alpha$ for all $(t,s)\in[0,T]\times \mathbb{R}$, hypothesis \rife{condmin2} is satisfied and the functional $I$ attains a minimum.
\end{remark}
{ Let us call 
$$\dys F_0=\inf_{(t,s)\in[0,T]\times \mathbb{R}}F(t,s),$$ and let us note that, by condition \rife{Flim}, it is finite. 
\begin{corollary}\label{cor2}
If, in addition to $(H_\Phi)$, $(H_f)$ and $(H_h)$, we assume that there exists a positive constant $k$ such that $\Phi(s)\geq ks^2$ for every $s\in[-a,a]$; and it holds
\begin{equation}\label{condpm2}
F_0T-\frac{T\Vert h\Vert^2_\infty }{4k}>m+\alpha T, 
\end{equation}
then $I$ has a critical point of mountain-pass nature.
\end{corollary}}
\begin{remark}
%\begin{enumerate}
% \item [(i)]B
%\item [(ii)] 
Observe that in the case of the relativistic model problem \rife{primoo} we have $\Phi(s)=1-\sqrt{1-s^2}, \ s\in[-1,1]$, which satisfies $(H_\Phi)$ and  $\Phi(s)\geq \frac{1}{4}s^2$.
%\end{enumerate}

\end{remark}
\medskip
}

\setcounter{equation}{0}
\section{Proof of the result}
Let us first recall an adapted version to Szulkin's theory of the Mountain-Pass
theorem, which we will use later. %This theorem provides, as its correspondent
%classical version due to Ambrosetti-Rabinowitz, the existence of a critical
%point under a geometric hypothesis and it we deal with a 

\begin{theorem}[Mountain Pass theorem]\label{saddle}
\emph{
Let $X$ be a Banach space and $I:X\rightarrow\mathbb{R}\cup\{+\infty\}$ be a functional decomposable as $I=J+\mathcal{F}$, with $J$ proper, convex and lower semicontinuous and $\mathcal{F}\in C^1$. 
%Let us consider $e\in X$ and let $W$ be the topological and algebraic complement of the space spanned by $e$. 
Let $W$ be a subspace of $X$ of codimension $1$ and suppose that there exist $u_1,u_2$ belonging to different connected components of $X\setminus W$ (notice that there are only two of them)  such that
\[
I(u_i)<\inf_{W}I,\quad i=1,2.
\] 
If $c$ is given by
\[
c=\inf_{\gamma\in\Gamma}\max_{x\in[0,1]} I(\gamma(x)),
\]
where $\Gamma=\{\gamma\in C([0,1],X): \gamma(0)=u_1, \ \gamma(1)=u_2\}$
and we assume that $(PS)_c$ holds, then $c$ is a critical value of $I$ (greater or equal than $\dys\inf_{W}I$).
%Suppose that $I$ satisfies  the $(PS)_c$-condition, and that $W_{\#}^{1,\infty}(0,T)=V\oplus W$, where $dim(V)<\infty$. If
%\begin{enumerate}
%\item[(sp1)] there exist costants $\rho>0$ and $\alpha_1$ such that $I\vert_{\partial B_\rho\cap V}\leq \alpha_1$,
%\item[(sp2)] there is a constant $\alpha_2>\alpha_1$ such that $I\vert_{W}\geq\alpha_2$,
%\end{enumerate}
%then $I$ has a critical point $c\geq\alpha_2$ which may be characterized by
%\[
%c=\inf_{g\in\Gamma}\sup_{x\in D} I(g(x)),
%\]
%where $D=\bar{B}_\rho\cap V$ and $\Gamma=\{g\in C(D,W_{\#}^{1,\infty}):g\vert_{\partial D}=id_{\partial D}\}.$}
}
\end{theorem}
\begin{proof}
This is a slight variant of Szulkin's version of Mountain Pass theorem (see \cite
[Theorem 3.4]{szul}).%, Theorem 3.5, page 92.
\end{proof}
We now prove that both $I$ and $J$ achieve their infimum value on $W$.
%\noindent {\bf Step 1. \emph{$I$ is bounded from below, i.e. } }
%\[
%\exists c \in \mathbb{R} \text{ such that }I(u)\geq c, \quad \forall u \in \hat{K}.
%\]
%We will prove that $I$ is bounded from below by showing that any single integral of $I$ is lower bounded in $\hat{K}$.
%\begin{enumerate}
%\item $\Phi$ is bounded from below. Indeed, since $\phi(0)=\Phi'(0)=0$ and $\phi=\Phi'$ is increasing, then
%\[ 
%\begin{cases}
%\Phi'(s)>0 \quad\forall s>0\\
%\Phi'(s)<0 \quad\forall s<0,
%\end{cases}
%\]
%so that $\Phi$ has a minimum at zero, let $\Phi(0)=c_0$. So we have:
%\[\Phi(s)\geq c_0, \quad \forall s\in [-a,a].\]
%
%\item 
%Since by $(H_f)$, $F(t,u)$ is bounded over $[0,T]\times\mathbb{R}$, then $\exists c_1\in\mathbb{R}$ such that $\vert F(t,u) \vert \leq c_1$ and so 
%\[F(t,u)  \geq - c_1, \quad \forall t\in [0,T].\]
%
%\item 
%As already noticed,
%$\vert u(t) \vert \leq \omega+Ta = c_2$, then:
%\[
%u(t)\geq -c_2, \quad \forall t\in [0,T].
%\]
%\end{enumerate}
%
%Thus:
%\begin{eqnarray*}
%I(u)&=&\int_0^T \big[\Phi(u'(t))+F(t,u(t))+h(t)u(t)\big]dt\\
%&\geq&c_0T-c_1T-c_2\Vert h\Vert_{L^1}=c>-\infty\,.
%\end{eqnarray*}
\begin{lemma}\label{minW}
The functional $I$ achieves its infimum value over the space $W$.
\end{lemma}
{\begin{proof}
Let us first of all observe that $I$ is bounded from below. Because of the definition of $I$, we only need to prove that it holds $I(v)\geq C\in \mathbb{R}$ when $v\in K$. In this case, it is easy to see that $\Vert \tilde{v}\Vert_\infty\leq aT$, using the mean value theorem and the fact that $\tilde{v}$ has zero mean value. Moreover, since $\Phi$ achieves its minimum at zero and $\Phi(0)=0$, $F$ is bounded and $h$ has zero mean value, it stands
$$\dys I(v)=\int_0^T [\Phi(\tilde{v}')+h(t)\tilde{v}+F(v)]dt\geq\dys \Big[-a\Vert h\Vert_{L^1}+\inf_{[0,T]\times\mathbb{R}}F(t,s)\Big]T>-\infty.$$

Our aim is to prove that if we call $\dys\beta=\inf_{\tilde{v}\in W}I(\tilde{v})$, then there exists $\tilde{z}\in W$ such that  $I(\tilde{z})= \beta$.

Let us consider a minimizing sequence $\{\tilde{z}_n\}\subset W\cap K$, i.e. satisfying that $\displaystyle\lim_{n\to\infty}I(\tilde{z}_n)=\beta$, and let us show that it is bouded. 
%Let us decompose any $z_n$ as in \rife{decomp} and 
As we already noticed, since any $\tilde{z}_n$ belongs to $K$, we have that $\Vert\tilde{z}_n\Vert_\infty\leq aT$ and that $\Vert\tilde{z}'_n\Vert_\infty\leq a$. Thus:
\[
\Vert \tilde{z}_n\Vert_{W_{\#}^{1,\infty}}=\Vert \tilde{z}_n\Vert_\infty+\Vert \tilde{z}'_n\Vert_\infty\leq (T+1)a.
\]
Applying now Ascoli-Arzel\`a theorem, we get the existence of a function $\tilde{z}\in C[0,T]$, such that up to a subsequence 
\[
\tilde{z}_n\to\tilde{z} \text{ in } C[0,T].
\]
Since $\tilde{z}_n\in W\cap K$, the uniform convergence  gives that $\tilde{z}\in W\cap K $. %, which ensures $\Vert \tilde{z}\Vert_{\infty} \leq Ta$.
%Notice that by  \rife{Flim} and since $\Phi$ is a bounded function, there exists a constant $c>0$ such that  $\Phi(\tilde{z}')+F(t,\tilde{z})>-c$. Moreover $\Vert \tilde{z}\Vert_\infty \leq Ta$, so that
%$$
%I(\tilde{z})\geq -cT-Ta \Vert h\Vert_{L^1}>-\infty
%$$
%it is easy to see that $I$ is bounded from below and let us call $\dys\gamma=\inf_{v\in W_{\#}^{1,\infty}}I(v)$.\\
 By the lower semicontinuity of $I$, we get
\[
\beta=\liminf_{n\to\infty}I(\tilde{z}_n)\geq I(\tilde{z}).
\]
It is straightforward $I(\tilde{z})=\beta$.% is a minimum of $I$ over the subspace $W$.
\end{proof}
}
\begin{remark}\label{Jbound}This lemma, with the particular choice $F\equiv 0$, shows that the same result holds for the functional $J$. In particular, %if we call $\dys m=\inf_{\tilde{v}\in W}J(\tilde{v})$, we get
\begin{equation}\label{JboundJ}
\exists \tilde{w}\in W\ \text{ such that } J(\tilde{w})= m=\min_{W}J.\bigskip
\end{equation}
\end{remark}

\begin{proof}[Proof of Theorem \ref{newcp}] 

%{We first prove that some $(PS)_c$-condition is satisfied for $I$ (Step 1), %then that the restriction of the functional $I$ to a suitable subspace $W$ of $ W_{\#}^{1,\infty}(0,T)$, always achieves its infimum value (Step 2); 
%we next observe that two different cases may occur, and that in both of them the existence of a critical point is guaranteed: by minimum in the first case (Step 2) and by mountain pass arguments in the second one (Step 3). In this last case, we need to define a translation of $I$ and prove that it satisfies the same $(PS)_c$-condition.}\medskip

\noindent {\bf Step 1. $I$ satisfies the $(PS)_c$-condition, for every level $c\neq m+\alpha T$.} 

Let us consider a $(PS)_c$-sequence in $K$. According to Definition \ref{psseq}, the following two conditions hold true:
\begin{enumerate}
\item[(PS1)] \label{cond1}$\displaystyle %I(u_n)= c +\eps_n,$ that is
\int_0^T \Big[\Phi(u'_n)+h(t)u_n+F(t,u_n)\Big]dt = c + \eps_n;
$
\item[(PS2)] $ \displaystyle \int_0^T [\Phi(v')-\Phi(u_n')]dt+\int_0^T\big[f(t,u_n)+h(t)\big](v-u_n)dt\geq-\varepsilon_n\Vert v-u_n\Vert_{W_{\#}^{1,\infty}} \,, \, \,$\\ $\forall v\in W_{\#}^{1,\infty}(0,T),
$  \label{cond2}
\end{enumerate}
where $\eps_n$ is a sequence converging to $0$.

We are interested in proving that $\{u_n\}$ has a uniformly convergent
subsequence. Let us note that we only need to show that $\{u_n\}$ is bounded in
$K$. Indeed by  Ascoli-Arzel\`a theorem, we can extract  a subsequence (not
relabeled) $\{u_{n}\}$ which converges uniformly in $[0,T]$ to a $u\in C[0,T]$. 
Moreover, since $\{u_n\}\subset K$ and thanks to the  uniform convergence, then $u\in K$.\smallskip

Let us decompose any $u_n=\tilde{u}_n+\bar{u}_n$ as in \rife{decomp}. We already
know that $\{\tilde{u}_n\}$ is bounded  in $K$ (see proof of Lemma \ref{minW}),
so that it only remains to show that $\{\bar{u}_n\}$ is bounded (notice
that it is a sequence of real numbers).

Let us suppose, by contradiction, that $\bar{u}_n$ diverges.
%\footnote{ Let $\tilde{w}$ be teh minimizer of $J$  on  W and let us choose 
Let $\tilde{w}$ be the minimizer of $J$ on $W$ given by \rife{JboundJ}
and let us choose $v=\tilde{w}+\bar{u}_n$ in (PS2). So we obtain
%Condition (PS2), %\rife{cond2} 
%with the choice $v=\tilde{w}+\bar{u}_n$ ($\tilde{w}$ the minimizer of $J$ on $W$) and multiplyed by $-1$, gives
%\[
%\displaystyle \int_0^T [\Phi(0)-\Phi(\tilde{u}_n'(t))]dt+\int_0^Tf(t,\tilde{u}_n(t)+\bar{u}_n(t))(-\tilde{u}_n(t))dt\geq-\varepsilon_n\Vert \tilde{u}_n\Vert_{\infty}\,, 
%\]
%that multiplying by $-1$,  it becomes
\[
\displaystyle \int_0^T \Big[\Phi(\tilde{u}_n')-\Phi(\tilde{w}')+h(t)\tilde{u}_n-h(t)\tilde{w}\Big]dt-\int_0^Tf(t,\tilde{u}_n+\bar{u}_n)(\tilde{w}-\tilde{u}_n)dt\leq\varepsilon_n\Vert \tilde{w}-\tilde{u}_n\Vert_{W_{\#}^{1,\infty}} .
\]
%We first observe that the argument of the first integral above is positive, since $\Phi$ achieves its minimum value in zero. 
Thus, taking in account the definition of $J$, recalling that $\eps_n$ vanishes as $n$ diverges and being $\Vert \tilde{w}-\tilde{u}_n\Vert_{W_{\#}^{1,\infty}} $ bounded, we deduce that (up to subsequences, not relabeled)
\[
 \limsup_{n\to\infty}\big[J(\tilde{u}_n)-J(\tilde{w})\big]\leq   \lim_{n\to\infty}\int_0^T f(t, \tilde{u}_n+\bar{u}_n )(\tilde{w}-\tilde{u}_n)dt.
\]
Thanks to \rife{Flim} we can use Lebesgue Theorem, and by $(H_f)$ the right
hand side above tends to zero, so that:
\begin{equation*}
 \limsup_{n\to\infty}\big[J(\tilde{u}_n)-J(\tilde{w})\big]\leq 0\,.
\end{equation*}
Consequently, since  $\dys J(\tilde{u}_n)\geq J(\tilde{w})=\min_{\tilde{v}\in W}J(\tilde{v})$ for all $n$, we get
%$J(\tilde{u}_n)-J(\tilde{w})\geq0$, and consequently
\begin{equation*}%\label{contradicion}
  \lim_{n\to\infty}J(\tilde{u}_n) =J(\tilde{w})  = m\,.
\end{equation*}
From this equality, thanks again to \rife{Flim}, using $(H_f)$ and up to subsequences, %(which allows us to use  Lebesgue theorem) 
we get
\begin{eqnarray*}
\lim_{n\to\infty} I(u_n)=       \lim_{n\to \infty} J (u_n) + \int_0^T F (\tilde{u} _n+ \bar{ u}_n) = m+\alpha T.%\label{2}
\end{eqnarray*}
On the other hand, by condition (PS1), we know that $\dys \lim_{n\to\infty} I(u_n)=c$.

Hence, the only constant $c$ which allows the sequence $\bar{u}_n$ to be
unbounded is $c= m+\alpha T$. %In every other case  the sequence is bounded and
%the $(PS)_c$-condition for $I$ holds.
In conclusion, we have proved that the $(PS)_c$-condition holds for every $c\neq m+\alpha T$.\medskip

%In particular, if $c\neq m+\alpha T$, any $(PS)_c$-sequence converges uniformly to a $u\in K$ and so $I(u)>-\infty$, as it can be seen using the same argument as in the proof of Lemma \ref{minW}.\medskip

Let us now consider the functional $I$ over the subspace $W$. Thanks to Lemma \ref{minW}, there exists $\tilde{z}\in W$ such that $\dys I(\tilde{z})=\min_{\tilde{v}\in W} I(\tilde{v})$: there are now two possibilities. Either
\begin{equation}\label{condmin}
%\exists \tilde{v} \in W\cap K \text{ such that } 
I(\tilde{z})\leq m+\alpha T,
\end{equation}
or
\begin{equation}\label{condpm}
I(\tilde{z})> m+\alpha T.
%\forall \tilde{u}\in W:\ I(\tilde{u})>(\Phi(0)+\alpha\big)T.
\end{equation}\bigskip
%\newpage
{
\noindent{\bf Step 2. If \rife{condmin} holds,  then $I$ has a minimum.}}

Condition \rife{condmin} implies that $\dys\gamma=\inf_{v\in K} I(v)\leq I(\tilde{z}) \leq m+\alpha T$. We prove that if $\gamma\leq m+\alpha T$ then the infimum of $I$ in $K$ is attained. 
%\[
%\gamma=\inf_{v\in K} I(v)\leq I(\tilde{z}) \leq m+\alpha T.
%\]

Indeed, two cases may occur

%\begin{itemize}
 {\it (Case 1.)}  If $\displaystyle\gamma=m+\alpha T$, it means that $I(\tilde{z})=m+\alpha T$. Hence %, so that 
 $\tilde{z}$ is a minimizer \\
 \indent\quad\quad\quad\quad of $I$ in  $K$ and we are done.  \smallskip
 
 {\it (Case 2.)}
 If $\displaystyle\gamma<m+\alpha T$, we know by Step 1 that $I$ satisfies the $(PS)_\gamma$-condition. Our first aim is 
%to show that there exists a minimizing sequence $\{u_n\}$, which is a $(PS)_\gamma$-sequence, in this wayi.e. $\{ u_n\}$ satisfies conditions (PS1) and (PS2).
%In order to do it, we 
to take a minimizing sequence $\{w_n\}\subset K$ and (via the use of Ekeland's
variational principle in \cite[Theorem 1]{ekel}) to construct from
$\{w_n\}$ a new sequence $\{u_n\}$ which is still
minimizing and that is also a $(PS)_\gamma$-sequence. 
%Let us recall, for the convenience of the reader, the following result, adapted to our framework.
%\begin{proposition}[Ekeland's variational principle]\label{ekeland}
%If $\exists w\in W_{\#}^{1,\infty}(0,T)$ and $\delta, \lambda>0$ such that 
%\[
%I(w)\leq\inf_{v\in W_{\#}^{1,\infty}}I(v)+\delta,
%\]
%then there exists a function $u\in W_{\#}^{1,\infty}(0,T)$ such that $I(u)\leq I(w)$, $\Vert w-u\Vert_{W_{\#}^{1,\infty}}\leq\frac{1}{\lambda}$ and
%\[
%I(z)-I(u)\geq -\delta\lambda\Vert z-u\Vert_{W_{\#}^{1,\infty}}, \quad \forall z\in W_{\#}^{1,\infty}(0,T).
%\]\medskip
%\end{proposition}
%Let us consider $\{w_n\}$, a minimizing sequence of $I$. 
By definition, it holds 
\[
I(w_n)=\gamma+\frac{1}{n}.
\]
{We apply Ekeland's principle 
(see \cite[Proposition 1.6]{szul} with the choice  $\delta=\frac{1}{n}, \lambda=1$) 
and we get the existence of  a sequence $\{u_n\}$ satisfying }
\begin{itemize}
\item $ \displaystyle \gamma\leq I(u_n)\leq I(w_n)= \gamma+\frac{1}{n},$
\item $\displaystyle I(z)-I(u_n)\geq-\frac{1}{n}\Vert z-u_n\Vert_{W_{\#}^{1,\infty}}, \qquad\forall z\in W_{\#}^{1,\infty}(0,T). $
\end{itemize}
Let us note that the first inequality implies (PS1) and  that $\{u_n\}$ is still a minimizing sequence. Let us now work on the second one to show that it leads to (PS2). \\
For any $v\in W_{\#}^{1,\infty}(0,T)$ and any $n\in\mathbb{N}$, we set $z=(1-\tau)u_n+\tau v, \, \tau \in(0,1)$, so that 
\[ 
I\big((1-\tau)u_n+\tau v\big)-I(u_n)\geq-\frac{1}{n}\tau \Vert v-u_n\Vert_{W_{\#}^{1,\infty}}.
\] 
On the other hand, using the convexity of $\Phi$, we get
$$\begin{array}{c}
\displaystyle I\big((1-\tau)u_n+\tau v\big)-I(u_n)\smallskip\\
%\displaystyle =\int_0^T\big[\Phi\big((1-\tau)u'_n+\tau v'\big)-\Phi(u'_n)\big]dt+\int_0^T[(1-\tau)h(t)u_n+\tau h(t)v-h(t)u_n]dt\\
%\dys+\int_0^T\big[F\big(t,(1-\tau)u_n+\tau v\big)-F(t,u_n)\big]dt\\
\displaystyle \leq \tau \int_0^T \big[\Phi(v')-\Phi(u'_n)+h(t)(v-u_n)\big]dt+\int_0^T\big[F\big(t,(u_n+\tau(v-u_n)\big)-F(t,u_n)\big]dt,
\end{array}$$
which yields
\[
\begin{array}{c}
\dys\tau \int_0^T \big[\Phi(v')-\Phi(u'_n)+h(t)(v-u_n)\big]dt+\int_0^T\big[F\big(t,(u_n+\tau (v-u_n)\big)-F(t,u_n)\big]dt\smallskip\\
\dys\geq-\frac{\tau}{n}\Vert v-u_n\Vert_{W_{\#}^{1,\infty}}\,.
\end{array}
\]
 Since $F$ is the bounded primitive of $f$ and we can use the mean value theorem, dividing by $\tau$ and letting $\tau \to0$, we get
\begin{eqnarray*}
\int_0^T\big[\Phi(v')-\Phi(u'_n)+ [f(t,u_n)+h(t)](v-u_n)\big]dt\geq-\frac{1}{n}\Vert v-u_n\Vert_{W_{\#}^{1,\infty}},\,\,\\
\forall v\in W_{\#}^{1,\infty}(0,T)\,,
\end{eqnarray*}
i.e. condition (PS2). %is the same we used in Step 1 of proof of Theorem \ref{minsol}, to switch from \rife{prec} to \rife{prec2}).\\
We have now found that $\{u_n\}$ is a minimizing sequence satisfying the $(PS)_\gamma$-condition with $\gamma<m+\alpha T$, so that, by Step 1, it is possible to extract a uniformly convergent subsequence (not relabeled) such that
\[
u_n\to u \text{\, uniformly on $[0,T]$,} \quad  u\in K.
\]
%Moreover $u\in K$, from the uniform convergence and since $u_n\in K$.\\
Using again the lower semicontinuity of $I$ and  $I(u)>-\infty$, we get that
\begin{eqnarray*}
\gamma=\liminf_{n\to\infty}I(u_n)\geq I(u).
\end{eqnarray*}
We have proved eventually that
$I(u)=\gamma$ and that $u$ is minimizer of $I$.\medskip
%\end{itemize}\smallskip

\noindent{\bf Step 3. If \rife{condpm} holds, then $I$ has a mountain pass.}\smallskip% critical point.}\\

We want now to apply Theorem \ref{saddle}.\smallskip %Hence we need to find suitable
%informations about the functional over $W$ and $\langle 1\rangle$.  

%Let us first consider $L$ over the space $W$ of the function with zero mean
%value. We stress that $v\in W$ if and only if $ v+\tilde{w}\in W$, which
%leads to
%$$
%\beta =\inf_{\tilde{v}\in W} I(\tilde{v})=\inf_{\tilde{v}\in W} L(\tilde{v}).
%$$
%Thus, by \rife{condpm}, we deduce that $\dys \inf_{\tilde{v}\in W}
%L(\tilde{v})>m+\alpha T$.

%By \rife{condpm}, we know that $\dys\beta=\inf_{\tilde{v}\in W} I(\tilde{v})>m+\alpha T$.

Let us recall that $\tilde{w}$ is the function such that %$\dys J(\tilde{w})=\inf_{\tilde{v}\in W}J(\tilde{v})$ 
$J(\tilde{w})=m$ and, by considering $n\in\mathbb{R}$ and since $h$ has zero mean value, we get
$$
I(\tilde{w}\pm n)=\int_0^T [\Phi(\tilde{w}')+h(t)\tilde{w}]dt +\int_0^T F (t,\tilde{w}\pm n) dt=m+\int_0^T F (t,\tilde{w}\pm n) dt.
$$
Using now $(H_f)$, we deduce that both
\[
\lim_{n\to+\infty}I(\tilde{w}+n)= m+\alpha T\quad \text{and}\quad\lim_{n\to+\infty}I(\tilde{w}-n)= m+\alpha T. 
\]
That is, $\forall \eps>0,$ there exists $n_0=n_0(\eps)\in\mathbb{N}$, such that 
\[
\Big\vert I(\tilde{w}+n)-\big(m+\alpha T\big) \Big\vert<\eps \quad  \text{ and } \quad  \Big\vert I(\tilde{w}-n)-\big(m+\alpha T\big)\Big\vert<\eps, \quad \forall n\geq n_0\,.
\]
 Hence, if we fix any
\[
0<\eps_1\leq\frac{\beta-\big(m+\alpha T\big)}{2},
\]
then we  find  $n_1\in \mathbb{N}$ such that
%\be \label{spdue}
\[
 I(\tilde{w}+n_1)-\big(m+\alpha T\big)\leq \eps_1 \quad  \text{ and } \quad    I(\tilde{w}-n_1)-\big(m+\alpha T\big)  \leq \eps_1.% , \quad \forall n\geq n_1\,.
\]
%\ee
By \rife{condpm}, we deduce
\[
I(\tilde{w}\pm n_1)\leq \eps_1+(m+\alpha T)<\beta.
\]
%We can apply Theorem \ref{saddle} with the choice $e=1$
%so that we have the geometry of the mountain pass theorem, which provides the critical point we wanted to exist.

%We know that $L$ satisfies the $(PS)_c$-condition {$\forall c\geq
%\beta>\big(m+\alpha T\big)$}. Hence we can apply Theorem \ref{saddle} with the  
%choices $V=\langle 1\rangle$ (notice that $dim(\langle 1\rangle)=1$) and  $W$.  
%Assumption (sp1) is satisfied: indeed choosing $\rho=n_1$  and by \rife{spdue}, we get 
%$$
%L\vert_{\partial B_\rho \cap \langle 1 \rangle }\leq \alpha_1=\eps_1+\big(m+\alpha T\big)\leq\frac{\beta+\big(m+\alpha T\big)}{2}<\beta.
%$$
%As far as assumption (sp2) is concerned, it follows by  condition  \rife{condpm} with the choice $\alpha_2=\beta$.  
Let us call 
\[
c=\inf_{\gamma\in\Gamma}\max_{x\in[0,1]} I(\gamma(x)),
\]
where $\Gamma=\{\gamma\in C([0,1],W_{\#}^{1,\infty} (0,T)): \gamma(0)=\tilde{w}-n_1, \ \gamma(1)=\tilde{w}+n_1\}$.
Since $c\geq\beta>m+\alpha T$ and by Step 1,  condition $(PS)_c$ holds and 
Theorem \ref{saddle} guarantees that $c$ is a  critical value of $I$. Consequenly there exists a function $u$ which is a critical point of $I$ of mountain pass nature.%, provided by:
%\[
%\inf_{\psi \in\Gamma}\sup_{x\in[-n_1,n_1]} L(\psi(x)),
%\]
%where $\Gamma=\{\psi:\in C([-n_1,n_1], W_{\#}^{1,\infty}) \text{ such that } \psi(n_1)=n_1, \,\,\psi (-n_1)=-n_1\}$.\bigskip
%
%Notice that once found a critical point of the functional $L$, we have automatically found one for $I$. Indeed if $\varsigma$ is a critical point of $L$, according to Definition \ref{defptocrit}, it holds
%\[  
%\int_0^T \big[\Phi(v')-\Phi(\tilde{w}'+\varsigma')+\big(f(t,\tilde{w}+\varsigma)+h(t)\big)(v-\tilde{w}-\varsigma)\big]dt\geq0, \, \,\forall v\in W_{\#}^{1,\infty}(0,T);
%\]
%calling $y=\varsigma+\tilde{w}$, it is straightforward that $y$ is a critical point of $I$.
\end{proof}\bigskip

In order to complete the proof of Theorem \ref{new}, we need the following result about the critical points of the functional $ I$.% defined by
%\[
%\mathcal{I}(v)=
%\begin{cases}
%\dys\int_0^T\big[\Phi(v')+h(t)v+G(t,v)\big]dt, \quad &\forall v\in K,\\
%+\infty &\forall v\in W_{\#}^{1,\infty}(0,T)\setminus K.
%\end{cases}
%\]
\begin{theorem}\label{strongexist}
\emph{
If we assume hypotheses $(H_\Phi)$, $(H_h)$ and $f$ is an essentially bounded function in $(0,T)\times\mathbb{R}$, then every critical point of ${I}$  is a solution of \rife{simplenew}.}
% \begin{equation}
%\begin{cases} \label{conm}
%(\phi(u'(t)))'=g(t,u(t))+h(t),\quad \text{ in } (0,T);\\
%u(0)-u(T)=0=u'(0)-u'(T),
%\end{cases}
%\end{equation}}
\end{theorem}
\begin{proof}
The proof follows the same strategy which is used in \cite{due}. %, but we will give here all the details.
Let us suppose that $u$ is a critical point of ${I}$. Then, according to Szulkin's theory, it satisfies
\begin{equation*}\label{ineqI}
\int_0^T\big[\Phi(v')-\Phi(u')+\big(f(t,u)+h(t)\big)(v-u)\big]dt\geq0,\quad\forall v\in K.
\end{equation*}
%Indeed, \rife{ineqI} is the variational inequality \rife{disvarcrit}, where the roles of $\mathcal{P}$ and $\mathcal{R}$ are respectively played  by
%\[
%\begin{array}{l}
%\displaystyle\mathcal{P}(v)=\int_0^T \Phi(v')dt,\quad \text{ and }\quad
%\displaystyle\mathcal{R}(v)=\int_0^T \big[G(t,u)+h(t)v\big]dt.
%\end{array}
%\]
In order to apply Corollary $3$ in \cite{gradotop}, we rewrite this inequality as %\rife{ineqI}, 
\begin{equation*}\label{ineqII}
\int_0^T\big[\Phi(v')-\Phi(u')+\big(l_u(t)+u\big)(v-u)\big]dt\geq0,\quad\forall v\in K.
\end{equation*}
where
\[
l_u(t)=f(t,u)+h(t)-u.
\]
According to \cite{due}, it is the variational inequality solved by the unique solution of problem
\begin{equation*}
\begin{cases} \label{ult}
(\phi(u'(t)))'=u(t)+l_u(t),\quad \text{ in } (0,T);\\
u(0)-u(T)=0=u'(0)-u'(T).
\end{cases}
\end{equation*}
We know then, that $u$ is the solution of the above problem and that
$u \in C^1(0,T)$, $ \phi\circ u' $ is absolutely continuous and $\Vert u \Vert _\infty < a$.
Moreover it holds
$
(\phi(u'))'=u+l_{{u}}(t)=f(t,{u})+h(t)
$ 
 and in conclusion, the critical point $u$ is a solution of \rife{simplenew}.% and it completes the proof.

\end{proof}

\begin{proof}[Proof of Theorem \ref{new}] Because of Theorem \ref{newcp}, we know that at least a critical point of $I$ exists. By Theorem \ref{strongexist}, it is guaranteed that such a point is a solution of \rife{simplenew}.
\end{proof}

{%\color{blue}
\begin{proof}[Proof of Corollary \ref{cor1}]
Let us note that if \rife{condmin2} is satisfied, then 
\[
I(\tilde{w}+\bar{v}_0)=\int_0^T[\Phi(\tilde{w}')+h(t)\tilde{w}]dt+\int_0^TF(t,\bar{v}_0+\tilde{w})dt\leq m+\alpha T,
\]
which implies that $\dys \inf_{v\in K} I(v)\leq m+\alpha T$. Consequently, arguing as in  Step 2 of the proof of Theorem \ref{new},  there exists a solution by minimization.
\end{proof}

}

\begin{proof}[Proof of Corollary \ref{cor2}]
Recall that $\dys F_0=\inf_{(t,s)\in[0,T]\times \mathbb{R}}F(t,s)$;
since $\Phi(s)\geq ks^2$, it holds 
\begin{eqnarray*}
I(\tilde{v})&=&\int_0^T [\Phi(\tilde{v}')+h(t)\tilde{v}+F(t,\tilde{v})]\\
&\geq&k\Vert \tilde{v}'\Vert^2_{L^2}-\Vert h\Vert_\infty\int_0^T\vert \tilde{v}\vert dt +F_0T,
\end{eqnarray*}
 for any $\tilde{v}\in W$. Notice that in $W$ Poincar\'e-Wirtinger inequality holds, i.e. $\int_0^T\tilde{v}^2dt\leq\int_0^T\tilde{v}'^2dt$. Thus, thanks also to H\"older inequality,  we deduce
\[
I(\tilde{v})\geq k\Vert \tilde{v}'\Vert^2_{L^2} -\sqrt{T}\Vert h\Vert_\infty\Vert \tilde{v}'\Vert_{L^2}+F_0T=p(\Vert \tilde{v}'\Vert_{L^2}),
\]
 where $p$ is the polynomial defined by $\dys p(s):=ks^2-\sqrt{T}\Vert h\Vert_\infty s+F_0T$. It is easy to see that $p$ attains its infimum value $\dys-\frac{T\Vert h\Vert^2_\infty}{4k}+F_0Ts=\frac{\sqrt{T}\Vert h\Vert_\infty}{2k}$, thus we deduce % Since $\dys p\Big(\frac{\Vert h\Vert_\infty}{2k}\Big)=-\frac{\Vert h\Vert^2_\infty}{4k}+T\inf_{[0,T]\times \mathbb{R}}F(t,s)$, then 
\[
\min_W I\geq -\frac{T\Vert h\Vert^2_\infty}{4k}+F_0T.
\]
Hence, if  \rife{condpm2} holds true, then 
\[
\min_{\tilde{v}\in W}I(\tilde{v})>m+\alpha T,
\] 
i.e. \rife{condpm} is verified. By Step 3 of the proof of Theorem \ref{new}, the existence of a mountain pass is guaranteed.% again by Theorem \ref{newcp}. %By Theorem \ref{new}, such a point is solution of Problem \rife{simplenew}.
\end{proof}

\begin{ack}
The author would like to thank David Arcoya and Tommaso Leonori, without their constant and patient supervision this work would not have been possible.
\end{ack}

%$0<\alpha \leq \beta$ such that 
%\begin{alignat}{4} \label{g2}
%f (s)\geq \alpha g(s)
%\qquad \forall s\geq 0\,,
%\tag{H1}
%\end{alignat}
%and
%\begin{alignat}{4} \label{g1}
%0< f (s)\leq \beta  g(s)
%\qquad \forall s\geq 0.
%\tag{H2}
%\end{alignat}


\begin{thebibliography}{99999}

%\bibitem[AmAr]{librodavid} A. Ambrosetti, D. Arcoya, \emph{An Introduction to Nonlinear Functional Analysis and Elliptic Problems}, Progress in Nonlinear Differential Equations and Their Applications, Vol. 82, BirkhŠuser Boston (2011)
\bibitem[AmCZ]{ambcz} A. Ambrosetti, V. Coti Zelati, \emph{Critical Points with Lack of Compactness and Singular Dynamical Systems}, Ann. Mat. Pura Appl. 149 (1987), 237--259.
\bibitem[AmRa]{ambrab} A. Ambrosetti, P. Rabinowitz, \emph{Dual variational methods in critical point theory and applications}, J. Funct. Anal. 14 (1973), 349--381.
%\bibitem[Ar]{arcart} D. Arcoya, \emph{Periodic solutions of Hamiltonian systems with strong resonance at infinity}, Differential Itegral Equations 3 (1990), 909--921.
%\bibitem[ArCa]{arcca} D. Arcoya,  A.Ca\~nada, \emph{Critical point theorems and applications to nonlinear boundary value problems}, Nonlinear Analysis, Theory, Methods \& applications, 14 (1990), 393--411.
\bibitem[Ar]{arcca} D. Arcoya,  \emph{Periodic solutions of Hamiltonian systems with strong resonance at infinity}, Differential Integral Equations, 3 (1990), 909--921.
%\bibitem[BeJM]{redial} C. Bereanu, P. Jebelean, J. Mawhin, \emph{Variational methods for nonlinear perturbations of singular $\phi$-Laplacians}, Rend. Lincei Mat. Appl. 22 (2011), 89--111.
\bibitem[BeM]{gradotop} C. Bereanu, J. Mawhin, \emph{Existence and multiplicity results for some nonlinear problems with singular $\phi$-Laplacian}, J. Differential Equations, 243  (2007), 536--557.
\bibitem[BeTo]{secsol} C. Bereanu, P. Torres, \emph{Existence of at least two periodic solutions of the forced relativistic pendulum}, %Differential Integral Equations, 23  (2010), 801-810. 
Proc. Am. Math. Soc., in press.
%\bibitem[Br]{brezis} H. Brezis, \emph{Functional Analysis, Sobolev Spaces and Partial Differential Equations}, Springer New York (2010).
\bibitem[BrM]{due} H. Brezis, J. Mawhin, \emph{Periodic solutions of the forced relativistic pendulum}, Differential Integral Equations, 23  (2010), 801--810.
%\bibitem[Cl]{subdiff} F. H. Clarke, \emph{Optimization and nonsmooth analysis}, 
%Canadian Mathematical Society Series of Monographs and Advanced Texts,  John Wiley \& Sons Inc., New York (1983). 
\bibitem[CZ]{cotizel} V. Coti Zelati, \emph{Periodic Solutions of dynamical Systems with Bounded Potential} J. Differential Equations, 67 (1987), 400--413. 
\bibitem[D]{dancer} N. Dancer, \emph{On the use of asymptotics in nonlinear boundary value problems}, Ann. Mat. Pura Appl. 131 (1982), 167--185.
\bibitem[Ek]{ekel} I. Ekeland, \emph{Nonconvex minimization problems}, Bull. Amer. Math. Soc., 1 (1979), 443--474.
%\bibitem[Ev]{evans} L. C. Evans, \emph{Partial Differential Equations}, Graduate Studies in Mathematics 19, American Mathematical Society (1998).

\bibitem[H]{Hamel} G. Hamel, \emph{Ueber erzwungene Schingungen bei endlischen Amplituden}, Math. Ann. 86 (1922), 1--13.
\bibitem[K]{fis} R. Katz, \emph{An introduction to the special theory of relativity}, Commission on College Physics (1964).
\bibitem[MWi]{mawwil} J. Mawhin, M. Willem, \emph{Multiple solutions of the periodic boundary value problem
for some forced pendulum-type equations}, J. Differential Equations, 52 (1984),
264--287.

%\bibitem[Ru] {rud}W. Rudin, \emph{Functional Analysis}, McGraw Hill, (1973).
\bibitem[Sz]{szul} A. Szulkin, \emph{Minmax principles for lower semicontinuous functions and applications to nonlinear boundary value problems}, Ann. Inst. H. Poincar\'e (1986), 77--109.
\bibitem[Th]{teps} K. Thews, \emph{T-periodc solutions of time dependent hamiltonian systems with a potential vanishing at infinity}, manuscripta math. 33 (1981), 327--338.
%{\bibitem[To1]{torres1} P. Torres, \emph{Nondegeneracy of the periodically forced linear differential equation with $\varphi$-Laplacian}, Comm. Cont. Math., to appear.
%\bibitem[To2]{torres2} P. Torres, \emph{Periodic oscillations of the relativistic pendulum with friction}, Phys. Lett. A 372 (2008), 6386--6387.}
\bibitem[Wa]{ward} J. R. Jr. Ward, \emph{A boundary value problem with a periodic nonlinearity}, Nonlinear Analysis, Theory, Methods \& applications,  10 (1986), 207--213.
\bibitem[Wi]{will} M. Willem, \emph{Oscillations forc\'ees de l'\'equation du pendule}, Publ. IRMA Lille, 3 (1981), V-1--V-3.


\end{thebibliography}
\end{document}